\documentclass[oneside,reqno,english]{amsart}
\usepackage[T1]{fontenc}
\usepackage[utf8]{inputenc}
\usepackage{xcolor}
\usepackage{babel}
\usepackage{prettyref}
\usepackage{mathrsfs}
\usepackage{amstext}
\usepackage{amsthm}
\usepackage{amssymb}
\usepackage[all]{xy}
\usepackage[pdfusetitle,
 bookmarks=true,bookmarksnumbered=false,bookmarksopen=false,
 breaklinks=false,pdfborder={0 0 0},pdfborderstyle={},backref=false,colorlinks=false]
 {hyperref}
\hypersetup{
 colorlinks=true,citecolor=blue,linkcolor=blue,linktocpage=true}

\makeatletter
\numberwithin{equation}{section}
\numberwithin{figure}{section}

\newrefformat{cor}{Corollary~\ref{#1}}
\newrefformat{subsec}{Section~\ref{#1}}
\newrefformat{lem}{Lemma~\ref{#1}}
\newrefformat{thm}{Theorem~\ref{#1}}
\newrefformat{sec}{Section~\ref{#1}}
\newrefformat{chap}{Chapter~\ref{#1}}
\newrefformat{prop}{Proposition~\ref{#1}}
\newrefformat{exa}{Example~\ref{#1}}
\newrefformat{tab}{Table~\ref{#1}}
\newrefformat{rem}{Remark~\ref{#1}}
\newrefformat{def}{Definition~\ref{#1}}
\newrefformat{fig}{Figure~\ref{#1}}
\newrefformat{claim}{Claim~\ref{#1}}
\newrefformat{assu}{Assumption~\ref{#1}}

\makeatother

\theoremstyle{plain}
\newtheorem{thm}{\protect\theoremname}[section]
\theoremstyle{remark}
\newtheorem{rem}[thm]{\protect\remarkname}
\theoremstyle{plain}
\newtheorem{prop}[thm]{\protect\propositionname}
\newtheorem{cor}[thm]{\protect\corollaryname}
\theoremstyle{definition}
\newtheorem{example}[thm]{\protect\examplename}
\providecommand{\corollaryname}{Corollary}
\providecommand{\examplename}{Example}
\providecommand{\propositionname}{Proposition}
\providecommand{\remarkname}{Remark}
\providecommand{\theoremname}{Theorem}

\begin{document}
\title{A Kernel Approach to the Stinespring and Kraus Representations}
\author{James Tian}
\begin{abstract}
We give a self-contained derivation of the Stinespring theorem for
completely positive maps and the Kraus representation for normal completely
positive maps using two scalar positive definite kernels. The first
kernel realizes the minimal Stinespring space directly as a reproducing
kernel Hilbert space, without passing through a quotient construction,
while the second separates the multiplicity space and produces the
Kraus operators as its coordinate operators. 

For holomorphic self-maps $f$ of the disk, this identifies the de
Branges-Rovnyak space $\mathcal{H}\left(f\right)$ as the canonical
Kraus multiplicity space of the associated map $\Phi_{f}$ on $B\left(H^{2}\right)$.
Composition of symbols is reflected by canonical isometries between
these multiplicity spaces, and the same kernel formulas describe the
iterates of $\Phi_{f}$. We show in particular that every $\Phi_{f}$
is extreme and that, when the iterates of $f$ converge to a point
of the disk, the preadjoint iterates converge in trace norm to the
corresponding Szegő kernel state. 
\end{abstract}

\address{Mathematical Reviews, 535 W. William St, Suite 210, Ann Arbor, MI
48103, USA}
\email{james.ftian@gmail.com}

\maketitle
\tableofcontents{}

\section{Introduction}\label{sec:1}

Completely positive maps form a fundamental class of maps between
algebras of bounded operators. Two classical results describe their
structure. Stinespring's theorem realizes a completely positive map
as a compression of a $*$-representation on a larger Hilbert space,
while the Kraus representation expresses a normal completely positive
map on $\mathscr{B}\left(H\right)$ as a finite or countable sum of
operator conjugations \cite{MR69403,MR292434,MR1976867}. The usual
Stinespring construction begins with a space of formal tensors, equips
it with a positive semidefinite form, quotients by the null space,
and then completes. In Sections \ref{sec:2}--\ref{sec:3}, we give
a different realization using two scalar positive definite kernels.
The first kernel is 
\[
K\left(\left(S,a\right),\left(T,b\right)\right)=\left\langle a,\varphi\left(S^{*}T\right)b\right\rangle .
\]
Its reproducing kernel Hilbert space gives a concrete realization
of the minimal Stinespring space. Because this space is realized from
the beginning as a space of functions on $\mathscr{B}\left(H\right)\times H$,
a vector of zero norm vanishes pointwise. Thus the identification
normally made by quotienting is incorporated into the kernel realization
itself. The first kernel gives the dilation space and its representation,
but it does not yet display the tensor product form behind the Kraus
representation. For a normal completely positive map, a second kernel,
formed from the matrix units of $\mathscr{B}\left(H\right)$, separates
the multiplicity space from the original Hilbert space. The resulting
kernel identity gives a unitary realization of the minimal Stinespring
space as 
\[
H\otimes H_{\widehat{K}},
\]
with the representation acting as $S\otimes I_{H_{\widehat{K}}}$.
Expanding the dilation operator in an orthonormal basis of $H_{\widehat{K}}$
then gives the Kraus operators as its coordinate operators. The role
of the two kernels is therefore different. The first realizes the
Stinespring space directly as a reproducing kernel Hilbert space,
while the second extracts its multiplicity space and converts the
abstract dilation into the Kraus representation. The argument uses
only scalar positive definite kernels and Hilbert space geometry.
In \prettyref{sec:4}, we apply the construction to holomorphic self-maps
of the disk. The norm estimate and strong row identity for the associated
weighted composition operators were proved in \cite{MR2336583}, while
endomorphisms and Cuntz families associated with inner functions were
studied in \cite{MR4456098}. For each holomorphic self-map $f:\mathbb{D}\rightarrow\mathbb{D}$,
the factorization 
\[
k\left(z,w\right)=k\left(f\left(z\right),f\left(w\right)\right)c_{f}\left(z,w\right)
\]
gives a normal unital completely positive map 
\[
\Phi_{f}:\mathscr{B}\left(H^{2}\right)\longrightarrow\mathscr{B}\left(H^{2}\right)
\]
satisfying 
\[
\Phi_{f}\left(M_{h}\right)=M_{h\circ f}.
\]
The two-kernel construction identifies the de Branges-Rovnyak space
$\mathcal{H}\left(f\right)$ as the minimal Kraus multiplicity space
of $\Phi_{f}$ and the weighted composition operators as its coordinate
Kraus operators. This gives the full operator formula 
\[
\Phi_{f}\left(S\right)=\sum_{\alpha}W_{u_{\alpha},f}SW^{*}_{u_{\alpha},f}
\]
in terms of weighted composition operators, with the previously known
row identity as the special case $S=I$. We then study the structure
of the family $\left\{ \Phi_{f}\right\} $. Every $\Phi_{f}$ is an
extreme unital completely positive map, and composition of symbols
gives 
\[
\Phi_{f}\circ\Phi_{g}=\Phi_{g\circ f}.
\]
At the level of multiplicity spaces, this relation is implemented
by canonical isometries 
\[
\mathcal{H}\left(g\circ f\right)\longrightarrow\mathcal{H}\left(g\right)\otimes\mathcal{H}\left(f\right).
\]
To our knowledge, this extremality result and the resulting composition
structure for general Schur functions have not appeared previously.
Finally, \prettyref{sec:5} develops a dynamical consequence of the
composition formula. Whenever the iterates $f^{\circ n}$ converge
locally uniformly to a point $a\in\mathbb{D}$, the preadjoint iterates
of $\Phi_{f}$ converge in trace norm to the normalized Szegő kernel
state at $a$.

\section{The Stinespring Dilation}\label{sec:2}

We begin by building the dilation space from a scalar kernel defined
directly by the CP map. The kernel will be shown to be positive definite
(p.d.), and its reproducing kernel Hilbert space (RKHS) will serve
as the dilation space. 

Throughout this section, let $H$ be a complex Hilbert space and $\mathscr{B}\left(H\right)$
the space of bounded operators on $H$. All inner products are linear
in the second variable. 
\begin{thm}[Stinespring dilation]
\label{thm:1} Let $\varphi:\mathscr{B}\left(H\right)\rightarrow\mathscr{B}\left(H\right)$
be a CP map. There exists a minimal Stinespring dilation $\left(\pi,H_{K},V\right)$
such that 
\[
\varphi\left(S\right)=V^{*}\pi\left(S\right)V
\]
where: 
\begin{enumerate}
\item $H_{K}$ is the RKHS associated with the $\mathbb{C}$-valued kernel
\[
K\left(\left(S,a\right),\left(T,b\right)\right):=\left\langle a,\varphi\left(S^{*}T\right)b\right\rangle .
\]
\item $V:H\rightarrow H_{K}$ is the bounded linear operator, 
\[
Vb:=K_{\left(I,b\right)}=K\left(\cdot,\left(I,b\right)\right).
\]
\item $\pi:\mathscr{B}\left(H\right)\rightarrow\mathscr{B}\left(H_{K}\right)$
is the unital $*$-representation defined on the dense span of sections
by
\[
\pi\left(S\right)K_{\left(T,b\right)}:=K_{\left(ST,b\right)}.
\]
\end{enumerate}
\end{thm}

\begin{proof}
Complete positivity of $\varphi$ ensures $K$ is a scalar-valued
positive definite (p.d.) kernel on $\left(\mathscr{B}\left(H\right)\times H\right)\times\left(\mathscr{B}\left(H\right)\times H\right)$.
Indeed, for any finite family $\left\{ \left(S_{i},a_{i}\right)\right\} ^{n}_{i=1}$
in $\mathscr{B}\left(H\right)\times H$, and $\left\{ c_{i}\right\} ^{n}_{i=1}$
in $\mathbb{C}$, setting $b_{i}=c_{i}a_{i}$, then 
\begin{align*}
\sum\nolimits_{i,j}\overline{c_{i}}c_{j}K\left(\left(S_{i},a_{i}\right),\left(S_{j},a_{j}\right)\right) & =\sum\nolimits_{i,j}\overline{c_{i}}c_{j}\left\langle a_{i},\varphi\left(S^{*}_{i}S_{j}\right)a_{j}\right\rangle \\
 & =\sum\nolimits_{i,j}\left\langle b_{i},\varphi\left(S^{*}_{i}S_{j}\right)b_{j}\right\rangle \\
 & =\left\langle \left(b_{1},\dots,b_{n}\right),\left[\varphi\left(S^{*}_{i}S_{j}\right)\right]\left(b_{1},\dots,b_{n}\right)\right\rangle \geq0
\end{align*}
because the block matrix $\left[\varphi\left(S^{*}_{i}S_{j}\right)\right]$
is positive in $M_{n}\left(\mathscr{B}\left(H\right)\right)$. 

Let $H_{K}$ be the corresponding RKHS with sections 
\[
K_{\left(T,b\right)}:=\left\langle \cdot,K\left(\cdot,T\right)b\right\rangle .
\]
By definition, 
\[
H_{K}=\overline{span}\left\{ K_{\left(T,b\right)}:T\in\mathscr{B}\left(H\right),\:b\in H\right\} .
\]
In particular, it consists of complex-valued functions on $\mathscr{B}\left(H\right)\times H$
with reproducing kernel $K$. See, e.g., \cite{MR51437,MR2938971,MR4250453}.

Define $V:H\rightarrow H_{K}$ by $Vb=K_{\left(I,b\right)}$, and
$\pi:\mathscr{B}\left(H\right)\rightarrow\mathscr{B}\left(H_{K}\right)$
by $\pi\left(S\right)K_{\left(T,b\right)}=K_{\left(ST,b\right)}$.
To see $\pi\left(S\right)$ is well defined and bounded, take a finite
sum $u=\sum_{k}c_{k}K_{\left(T_{k},b_{k}\right)}$. Then 
\begin{align*}
\left\Vert \pi\left(S\right)u\right\Vert ^{2} & =\Big\Vert\sum\nolimits_{k}c_{k}K_{\left(ST_{k},b_{k}\right)}\Big\Vert^{2}\\
 & =\sum\nolimits_{k,l}\overline{c_{k}}c_{l}K\left(\left(ST_{k},b_{k}\right),\left(ST_{l},b_{l}\right)\right)\\
 & =\sum\nolimits_{k,l}\overline{c_{k}}c_{l}\left\langle b_{k},\varphi\left(T^{*}_{k}S^{*}ST_{l}\right)b_{l}\right\rangle .
\end{align*}
Since $S^{*}S\leq\left\Vert S\right\Vert ^{2}I$ in $\mathscr{B}\left(H\right)$
and the map $A\mapsto T^{*}_{k}AT_{k}$ is positive, we have 
\[
\left[T^{*}_{k}S^{*}ST_{l}\right]\leq\left\Vert S\right\Vert ^{2}\left[T^{*}_{k}T_{l}\right]
\]
in $M_{n}\left(\mathscr{B}\left(H\right)\right)$. 

Applying the CP map $id_{M_{n}}\otimes\varphi$ and pairing with $\left(b_{k}\right)\in H^{\oplus n}$
yields 
\[
\sum\nolimits_{k,l}\overline{c_{k}}c_{l}\left\langle b_{k},\varphi\left(T^{*}_{k}S^{*}ST_{l}\right)b_{l}\right\rangle \leq\left\Vert S\right\Vert ^{2}\sum\nolimits_{k,l}\overline{c_{k}}c_{l}\left\langle b_{k},\varphi\left(T^{*}_{k}T_{l}\right)b_{l}\right\rangle =\left\Vert S\right\Vert ^{2}\left\Vert u\right\Vert ^{2}.
\]
Hence $\left\Vert \pi\left(S\right)\right\Vert \leq\left\Vert S\right\Vert $.
One checks that 
\[
\pi\left(S_{1}\right)\pi\left(S_{2}\right)=\pi\left(S_{1}S_{2}\right)
\]
on the dense set of sections. Since $K$ is Hermitian, $\pi\left(S\right)^{*}=\pi\left(S^{*}\right)$.
Hence $\pi$ is a unital $*$-representation of $\mathscr{B}\left(H\right)$. 

For all $S\in\mathscr{B}\left(H\right)$ and $a,b\in H$, 
\[
\left\langle a,\varphi\left(S\right)b\right\rangle =K\left(\left(I,a\right),\left(S,b\right)\right)=\left\langle Va,\pi\left(S\right)Vb\right\rangle _{H_{K}},
\]
that is, $\varphi\left(S\right)=V^{*}\pi\left(S\right)V$. 

Finally, the dilation is minimal by construction, as the spanning
set for $H_{K}$ is precisely $\left\{ K_{(S,a)}\right\} =\left\{ \pi\left(S\right)Va\right\} $.
\end{proof}

\begin{rem}
\label{rem:2-2}We work from the outset with the linear span of the
kernel sections as functions on $\mathscr{B}\left(H\right)\times H$.
If $u$ in this span has zero norm, then 
\[
\left|u\left(S,a\right)\right|=\left|\left\langle K_{\left(S,a\right)},u\right\rangle _{H_{K}}\right|\leq\left\Vert K_{\left(S,a\right)}\right\Vert _{H_{K}}\left\Vert u\right\Vert _{H_{K}}=0
\]
for every $\left(S,a\right)\in\mathscr{B}\left(H\right)\times H$.
Thus $u$ is the zero function. 

In the usual GNS-Stinespring construction, one first quotients by
the null space of the underlying sesquilinear form and then completes.
Here the realization by kernel functions incorporates that identification
automatically, so the quotient never appears as a separate step.
\end{rem}

\section{The Kraus Representation}\label{sec:3}

The dilation obtained in the previous section is still abstract. It
lives inside the RKHS generated by the first kernel. Our next goal
is to uncover the concrete tensor product structure hidden inside
that space. This structure is responsible for the familiar form of
a completely positive map as a sum of operator conjugations.

To reveal it, we introduce a second scalar positive definite kernel,
built from the matrix units of the underlying Hilbert space. This
auxiliary kernel generates a new reproducing kernel Hilbert space
which plays the role of the “ancillary’’ tensor factor.

A simple computation then shows that the original dilation space $H_{K}$
is unitarily equivalent to an honest tensor product of Hilbert spaces.
Once this identification is made, the Kraus representation follows
immediately by expanding the dilation operator along an orthonormal
basis of the auxiliary space.
\begin{thm}[canonical tensor structure]
\label{thm:3-1} Let $\varphi:\mathscr{B}\left(H\right)\rightarrow\mathscr{B}\left(H\right)$
be a normal CP map, and let $\left(\pi,H_{K},V\right)$ be the minimal
Stinespring dilation from \prettyref{thm:1}. Fix an orthonormal basis
(ONB) $\left\{ e_{i}\right\} $ in $H$, and let $E_{ij}=\left|e_{i}\left\rangle \right\langle e_{j}\right|$.
Let $H_{\hat{K}}$ be the RKHS of the kernel 
\[
\hat{K}\left(\left(i,a\right),\left(j,b\right)\right):=\left\langle a,\varphi\left(E_{ij}\right)b\right\rangle .
\]
Then the Stinespring space $H_{K}$ is canonically unitary to $H\otimes H_{\hat{K}}$
via an isomorphism $U$ that identifies:
\begin{enumerate}
\item The representation $\pi\left(S\right)$ with $S\otimes I_{H_{\hat{K}}}$. 
\item The operator $V$ with $W:H\rightarrow H\otimes H_{\hat{K}}$, where
\[
Wa:=\sum\nolimits_{i}e_{i}\otimes\hat{K}_{\left(i,a\right)}.
\]
\end{enumerate}
This identification transforms the abstract Stinespring form $\varphi(S)=V^{*}\pi(S)V$
into the canonical tensor product form 
\[
\varphi(S)=W^{*}\left(S\otimes I_{H_{\hat{K}}}\right)W.
\]
Furthermore, expanding $W$ in an ONB $\{f_{\alpha}\}$ of $H_{\hat{K}}$
yields the Kraus form
\[
\varphi\left(S\right)=\sum V^{*}_{\alpha}SV_{\alpha},
\]
where 
\[
V_{\alpha}:=\left(I_{H}\otimes\left\langle f_{\alpha},\cdot\right\rangle \right)W.
\]

\end{thm}

\begin{proof}
The proof proceeds in two steps. First, we establish the unitary isomorphism
$U$. Second, we transport the representation $(\pi,V)$ along $U$
to derive the tensor product and Kraus forms. See the diagram below.
\[
\xymatrix{H\ar[d]_{\varphi\left(S\right)}\ar[rr]^{V}\ar@/^{2pc}/[rrrr]^{W} &  & H_{K}\ar[rr]^{U}\ar[d]_{\pi\left(S\right)} &  & H\otimes H_{\hat{K}}\ar[d]^{S\otimes I_{H_{\hat{K}}}}\\
H\ar[rr]_{V}\ar@/_{2pc}/[rrrr]_{W} &  & H_{K}\ar[rr]_{U} &  & H\otimes H_{\hat{K}}
}
\]

\textbf{Step 1.} Fix an ONB $\left\{ e_{i}\right\} _{i\in I}$ in
$H$, and let $E_{ij}=\left|e_{i}\left\rangle \right\langle e_{j}\right|$.
The block matrix $[\varphi(E_{ij})]$ is positive by complete positivity,
so $\hat{K}$ is a scalar-valued p.d. kernel defined on $\left(I\times H\right)\times\left(I\times H\right)$.
Let $H_{\hat{K}}$ be the associated RKHS with sections $\hat{K}_{\left(i,a\right)}$,
so that 
\[
H_{\hat{K}}=\overline{span}\mathop{\{}\hat{K}_{\left(i,a\right)}:i\in I,\:a\in H\mathop{\}}.
\]

We now relate the kernel $K$ from \prettyref{thm:1} to this new
kernel $\hat{K}$. (When $H$ is infinite dimensional, the sums below
are understood as limits over finite subsets of the fixed orthonormal
basis. Their norm convergence and the passage through $\varphi$ follow
from the normality of $\varphi$.) For any $\left(S,a\right)$ and
$\left(T,b\right)$, 
\begin{align*}
K\left(\left(S,a\right),\left(T,b\right)\right) & =\left\langle a,\varphi\left(S^{*}T\right)b\right\rangle \\
 & =\sum\nolimits_{i,j}\left\langle e_{i},S^{*}Te_{j}\right\rangle \left\langle a,\varphi\left(E_{ij}\right)b\right\rangle \\
 & =\sum\nolimits_{i,j}\left\langle Se_{i},Te_{j}\right\rangle \left\langle \hat{K}_{\left(i,a\right)},\hat{K}_{\left(j,b\right)}\right\rangle _{H_{\hat{K}}}\\
 & =\left\langle \sum\nolimits_{i}Se_{i}\otimes\hat{K}_{\left(i,a\right)},\sum\nolimits_{j}Te_{j}\otimes\hat{K}_{\left(j,b\right)}\right\rangle _{H\otimes H_{\hat{K}}}.
\end{align*}
This calculation shows that the linear map $U:H_{K}\rightarrow H\otimes H_{\hat{K}}$,
defined on the spanning set of sections by
\[
U\left(K_{\left(T,b\right)}\right):=\sum\nolimits_{i}Te_{i}\otimes\hat{K}_{\left(i,b\right)}
\]
is isometric, with dense range in $H\otimes H_{\hat{K}}$. Thus, $U$
is unitary, and 
\[
H_{K}\simeq H\otimes H_{\hat{K}}.
\]

\textbf{Step 2.} We now transport the Stinespring representation $(\pi,V)$
from $H_{K}$ to $H\otimes H_{\hat{K}}$. 

Using $U$, we get 
\begin{align*}
U\pi\left(S\right)U^{*}\left(\sum\nolimits_{i}Te_{i}\otimes\hat{K}_{\left(i,b\right)}\right) & =U\left(\pi\left(S\right)K_{\left(T,b\right)}\right)\\
 & =U\left(K_{\left(ST,b\right)}\right)\\
 & =\sum\nolimits_{i}STe_{i}\otimes\hat{K}_{\left(i,b\right)}\\
 & =\left(S\otimes I\right)\left(\sum\nolimits_{i}Te_{i}\otimes\hat{K}_{\left(i,b\right)}\right).
\end{align*}
So $\pi\left(S\right)$ is unitarily equivalent to $S\otimes I_{H_{\hat{K}}}$. 

Also, the operator $V$ becomes $W:=UV$:
\[
Wa:=U\left(Va\right)=U\left(K_{\left(I,a\right)}\right)=\sum\nolimits_{i}e_{i}\otimes\hat{K}_{\left(i,a\right)}.
\]
Therefore, 
\begin{align*}
\varphi\left(S\right) & =V^{*}\pi\left(S\right)V=\left(UV\right)^{*}\left(U\pi\left(S\right)U^{*}\right)\left(UV\right)\\
 & =W^{*}\left(S\otimes I\right)W.
\end{align*}

Lastly, let $\left\{ f_{\alpha}\right\} _{\alpha\in\Lambda}$ be an
ONB of $H_{\hat{K}}$. Define the operators $V_{\alpha}:H\rightarrow H$
by 
\[
V_{\alpha}:=\left(I_{H}\otimes\left\langle f_{\alpha},\cdot\right\rangle _{H_{\hat{K}}}\right)W.
\]
This means $V_{\alpha}a$ is the vector component of $Wa$ along $f_{\alpha}$,
so 
\[
Wa=\sum\nolimits_{\alpha}\left(V_{\alpha}a\right)\otimes f_{\alpha}.
\]
This can be written as $W=\sum_{\alpha}V_{\alpha}\otimes f_{\alpha}$.
Substituting this into the tensor form gives:
\begin{align*}
\varphi\left(S\right) & =W^{*}\left(S\otimes I\right)W\\
 & =\left(\sum\nolimits_{\alpha}V^{*}_{\alpha}\otimes f_{\alpha}\right)\left(S\otimes I\right)\left(\sum\nolimits_{\beta}V_{\beta}\otimes f_{\beta}\right)\\
 & =\sum\nolimits_{\alpha,\beta}V^{*}_{\alpha}SV_{\beta}\left\langle f_{\alpha},f_{\beta}\right\rangle \\
 & =\sum\nolimits_{\alpha}V^{*}_{\alpha}SV_{\alpha}.
\end{align*}
The convergence is in the strong operator topology.
\end{proof}

\section{Schur functions and de Branges-Rovnyak spaces}\label{sec:4}

We now consider a setting in which the two kernels carry analytic
structure. Let $H^{2}$ be the Hardy space on the unit disk $\mathbb{D}$,
with Szegő kernel 
\[
k\left(z,w\right)=\frac{1}{1-z\overline{w}}.
\]
We write $k_{w}=k\left(\cdot,w\right)$, so that 
\[
h\left(w\right)=\left\langle k_{w},h\right\rangle _{H^{2}}
\]
for $h\in H^{2}$.

Let $\mathcal{S}\left(\mathbb{D}\right)$ denote the holomorphic self-maps
of $\mathbb{D}$. Fix $f\in\mathcal{S}\left(\mathbb{D}\right)$. Since
$f\in H^{\infty}$ and $\left\Vert f\right\Vert _{\infty}\leq1$,
multiplication by $f$ defines a contraction $M_{f}$ on $H^{2}$.
Its de Branges-Rovnyak kernel is 
\[
c_{f}\left(z,w\right)=\frac{1-f\left(z\right)\overline{f\left(w\right)}}{1-z\overline{w}}.
\]
This is a positive definite kernel. Indeed, if $M_{f}$ denotes multiplication
by $f$ on $H^{2}$, then 
\[
c_{f}\left(z,w\right)=\left\langle k_{z},\left(I-M_{f}M^{*}_{f}\right)k_{w}\right\rangle _{H^{2}}.
\]

We denote its RKHS by $\mathcal{H}\left(f\right)$ and write 
\[
c_{f,w}=c_{f}\left(\cdot,w\right).
\]
The two kernels are related by 
\begin{equation}
k\left(z,w\right)=k\left(f\left(z\right),f\left(w\right)\right)c_{f}\left(z,w\right).\label{eq:4-1}
\end{equation}

We refer to \cite{MR1289670} for the basic theory of these spaces.

\subsection{The associated completely positive map}

The identity \prettyref{eq:4-1} gives a canonical way to turn composition
on kernel vectors into a normal unital completely positive map on
$B\left(H^{2}\right)$.
\begin{thm}
\label{thm:4-1} The linear map defined on the kernel functions of
$H^{2}$ by 
\[
V_{f}k_{w}=k_{f\left(w\right)}\otimes c_{f,w}
\]
extends to an isometry 
\[
V_{f}:H^{2}\longrightarrow H^{2}\otimes\mathcal{H}\left(f\right).
\]
The formula 
\begin{equation}
\Phi_{f}\left(S\right)=V^{*}_{f}\left(S\otimes I_{\mathcal{H}\left(f\right)}\right)V_{f},\qquad S\in B\left(H^{2}\right),\label{eq:4-2}
\end{equation}
defines a normal unital completely positive map on $B\left(H^{2}\right)$.
The Stinespring representation in \prettyref{eq:4-2} is minimal.

Moreover, for every $h\in H^{\infty}$, 
\begin{equation}
\Phi_{f}\left(M_{h}\right)=M_{h\circ f}.\label{eq:4-3}
\end{equation}
If $\Phi_{f,*}$ denotes the preadjoint of $\Phi_{f}$ and 
\[
\rho_{w}=\left|\widetilde{k}_{w}\right\rangle \left\langle \widetilde{k}_{w}\right|,\qquad\widetilde{k}_{w}=\frac{k_{w}}{\left\Vert k_{w}\right\Vert },
\]
then 
\begin{equation}
\Phi_{f,*}\left(\rho_{w}\right)=\rho_{f\left(w\right)}.\label{eq:4-4}
\end{equation}
\end{thm}

\begin{proof}
For $z,w\in\mathbb{D}$, \prettyref{eq:4-1} gives 
\[
\begin{aligned}\left\langle V_{f}k_{z},V_{f}k_{w}\right\rangle  & =\left\langle k_{f\left(z\right)},k_{f\left(w\right)}\right\rangle _{H^{2}}\left\langle c_{f,z},c_{f,w}\right\rangle _{\mathcal{H}\left(f\right)}\\
 & =k\left(f\left(z\right),f\left(w\right)\right)c_{f}\left(z,w\right)\\
 & =k\left(z,w\right).
\end{aligned}
\]
Since the kernel functions span a dense subspace of $H^{2}$, the
map $V_{f}$ extends to an isometry. It follows that \prettyref{eq:4-2}
defines a normal unital completely positive map.

Let $\mathcal{L}$ be the closed linear span of 
\[
\left\{ \left(S\otimes I_{\mathcal{H}\left(f\right)}\right)V_{f}h:S\in B\left(H^{2}\right),\ h\in H^{2}\right\} .
\]
For fixed $w\in\mathbb{D}$ and $x\in H^{2}$, choose $S\in B\left(H^{2}\right)$
such that 
\[
Sk_{f\left(w\right)}=x.
\]
Then 
\[
x\otimes c_{f,w}=\left(S\otimes I_{\mathcal{H}\left(f\right)}\right)V_{f}k_{w}\in\mathcal{L}.
\]
The vectors $c_{f,w}$ span a dense subspace of $\mathcal{H}\left(f\right)$,
and hence 
\[
\mathcal{L}=H^{2}\otimes\mathcal{H}\left(f\right).
\]
The Stinespring representation is therefore minimal.

For $h\in H^{\infty}$, 
\[
\begin{aligned}\left(M^{*}_{h}\otimes I_{\mathcal{H}\left(f\right)}\right)V_{f}k_{w} & =\overline{h\left(f\left(w\right)\right)}k_{f\left(w\right)}\otimes c_{f,w}\\
 & =V_{f}M^{*}_{h\circ f}k_{w}.
\end{aligned}
\]
Thus 
\[
\left(M^{*}_{h}\otimes I_{\mathcal{H}\left(f\right)}\right)V_{f}=V_{f}M^{*}_{h\circ f},
\]
which gives \prettyref{eq:4-3}.

Finally, \prettyref{eq:4-1} gives 
\[
\left\Vert k_{w}\right\Vert ^{2}=\left\Vert k_{f\left(w\right)}\right\Vert ^{2}\left\Vert c_{f,w}\right\Vert ^{2}.
\]
Hence 
\[
V_{f}\widetilde{k}_{w}=\widetilde{k}_{f\left(w\right)}\otimes\frac{c_{f,w}}{\left\Vert c_{f,w}\right\Vert }.
\]
For $S\in B\left(H^{2}\right)$, it follows that 
\[
\left\langle \widetilde{k}_{w},\Phi_{f}\left(S\right)\widetilde{k}_{w}\right\rangle =\left\langle \widetilde{k}_{f\left(w\right)},S\widetilde{k}_{f\left(w\right)}\right\rangle .
\]
This proves \prettyref{eq:4-4}. 
\end{proof}

The action \prettyref{eq:4-4} on the normalized Szegő kernel states
actually characterizes the map among normal unital completely positive
maps. We show this below. 
\begin{prop}
\label{prop:4-2} Let $f:\mathbb{D}\to\mathbb{D}$ be holomorphic,
and let 
\[
\Psi:B\left(H^{2}\right)\longrightarrow B\left(H^{2}\right)
\]
be a normal unital completely positive map. If 
\[
\Psi_{*}\left(\rho_{w}\right)=\rho_{f\left(w\right)}
\]
for every $w\in\mathbb{D}$, then 
\[
\Psi=\Phi_{f}.
\]
\end{prop}

\begin{proof}
By \prettyref{thm:3-1}, there are a Hilbert space $\mathcal{E}$
and an isometry 
\[
W:H^{2}\longrightarrow H^{2}\otimes\mathcal{E}
\]
such that 
\[
\Psi\left(S\right)=W^{*}\left(S\otimes I_{\mathcal{E}}\right)W
\]
for every $S\in B\left(H^{2}\right)$.

Fix $w\in\mathbb{D}$. The assumption implies that the restriction
of the vector state associated with $W\widetilde{k}_{w}$ to the first
tensor factor is the pure state $\rho_{f\left(w\right)}$. It follows
that there is a unit vector $\eta_{w}\in\mathcal{E}$ such that 
\[
W\widetilde{k}_{w}=\widetilde{k}_{f\left(w\right)}\otimes\eta_{w}.
\]
Thus there is a vector $\xi_{w}\in\mathcal{E}$ such that 
\[
Wk_{w}=k_{f\left(w\right)}\otimes\xi_{w}.
\]

Since $W$ is an isometry, for $z,w\in\mathbb{D}$ we have 
\[
\begin{aligned}k\left(z,w\right) & =\left\langle Wk_{z},Wk_{w}\right\rangle \\
 & =k\left(f\left(z\right),f\left(w\right)\right)\left\langle \xi_{z},\xi_{w}\right\rangle .
\end{aligned}
\]
The Szegő kernel does not vanish on $\mathbb{D}\times\mathbb{D}$,
so \prettyref{eq:4-1} gives 
\[
\left\langle \xi_{z},\xi_{w}\right\rangle =c_{f}\left(z,w\right).
\]

It follows that, for every $S\in B\left(H^{2}\right)$, 
\[
\begin{aligned}\left\langle k_{z},\Psi\left(S\right)k_{w}\right\rangle  & =\left\langle Wk_{z},\left(S\otimes I_{\mathcal{E}}\right)Wk_{w}\right\rangle \\
 & =\left\langle k_{f\left(z\right)},Sk_{f\left(w\right)}\right\rangle c_{f}\left(z,w\right)\\
 & =\left\langle k_{z},\Phi_{f}\left(S\right)k_{w}\right\rangle .
\end{aligned}
\]
Since the kernel functions span a dense subspace of $H^{2}$, we have
\[
\Psi\left(S\right)=\Phi_{f}\left(S\right)
\]
for every $S\in B\left(H^{2}\right)$. 
\end{proof}

\subsection{Kraus structure and composition}

The construction above gives a Stinespring representation directly
from the factorization of the Szegő kernel. We now apply the two-kernel
method to this map, and show that the de Branges-Rovnyak space is
the multiplicity space selected by the second kernel. 

Let 
\[
K_{f}\left(\left(S,a\right),\left(T,b\right)\right)=\left\langle a,\Phi_{f}\left(S^{*}T\right)b\right\rangle ,
\]
and let $\mathcal{H}_{K_{f}}$ be its RKHS. Write 
\[
e_{n}\left(z\right)=z^{n},\qquad n\geq0,
\]
and define 
\[
E_{mn}h=e_{m}\left\langle e_{n},h\right\rangle _{H^{2}}.
\]
The second kernel associated with $\Phi_{f}$ is 
\[
\widehat{K}_{f}\left(\left(m,a\right),\left(n,b\right)\right)=\left\langle a,\Phi_{f}\left(E_{mn}\right)b\right\rangle .
\]
We denote its RKHS by $\mathcal{H}_{\widehat{K}_{f}}$.
\begin{prop}
\label{prop:4-3} The map defined on kernel sections by 
\[
\mathcal{U}_{f}K_{f,\left(S,a\right)}=\left(S\otimes I_{\mathcal{H}\left(f\right)}\right)V_{f}a
\]
extends to a unitary 
\[
\mathcal{U}_{f}:\mathcal{H}_{K_{f}}\longrightarrow H^{2}\otimes\mathcal{H}\left(f\right).
\]

For $n\geq0$ and $a\in H^{2}$, put 
\[
\gamma_{n,a}=\left(\left\langle e_{n},\cdot\right\rangle _{H^{2}}\otimes I_{\mathcal{H}\left(f\right)}\right)V_{f}a.
\]
The map 
\[
\widehat{\mathcal{U}}_{f}\widehat{K}_{f,\left(n,a\right)}=\gamma_{n,a}
\]
extends to a unitary 
\[
\widehat{\mathcal{U}}_{f}:\mathcal{H}_{\widehat{K}_{f}}\longrightarrow\mathcal{H}\left(f\right).
\]
In particular, 
\begin{equation}
\gamma_{n,k_{w}}=\overline{f\left(w\right)}^{n}c_{f,w}.\label{eq:4-5}
\end{equation}
\end{prop}

\begin{proof}
For $S,T\in B\left(H^{2}\right)$ and $a,b\in H^{2}$, 
\[
\left\langle \left(S\otimes I\right)V_{f}a,\left(T\otimes I\right)V_{f}b\right\rangle =\left\langle a,V^{*}_{f}\left(S^{*}T\otimes I\right)V_{f}b\right\rangle =K_{f}\left(\left(S,a\right),\left(T,b\right)\right).
\]
Thus $\mathcal{U}_{f}$ is isometric. Its range is dense by the minimality
in \prettyref{thm:4-1}, so it is unitary.

The definition of $\gamma_{n,a}$ gives 
\[
V_{f}a=\sum_{n\geq0}e_{n}\otimes\gamma_{n,a}.
\]
Therefore 
\[
\begin{aligned}\left\langle \gamma_{m,a},\gamma_{n,b}\right\rangle _{\mathcal{H}\left(f\right)} & =\left\langle V_{f}a,\left(E_{mn}\otimes I\right)V_{f}b\right\rangle \\
 & =\left\langle a,\Phi_{f}\left(E_{mn}\right)b\right\rangle \\
 & =\widehat{K}_{f}\left(\left(m,a\right),\left(n,b\right)\right).
\end{aligned}
\]
Hence $\widehat{\mathcal{U}}_{f}$ is isometric.

For $w\in\mathbb{D}$, 
\[
V_{f}k_{w}=k_{f\left(w\right)}\otimes c_{f,w}=\sum_{n\geq0}\overline{f\left(w\right)}^{n}e_{n}\otimes c_{f,w}.
\]
This proves \prettyref{eq:4-5}. In particular, 
\[
\gamma_{0,k_{w}}=c_{f,w}.
\]
Since the vectors $c_{f,w}$ span $\mathcal{H}\left(f\right)$, the
range of $\widehat{\mathcal{U}}_{f}$ is dense. It is therefore unitary. 
\end{proof}

Thus the first kernel recovers the full minimal Stinespring space
$H^{2}\otimes\mathcal{H}\left(f\right)$, and the second kernel selects
$\mathcal{H}\left(f\right)$ as its Kraus multiplicity space. In particular,
the Kraus operators obtained by taking coordinates in the second factor
are the weighted composition operators considered below.

For $u\in\mathcal{H}\left(f\right)$, consider the weighted composition
operator 
\[
W_{u,f}h=u\left(h\circ f\right).
\]

The norm estimate \prettyref{eq:4-6-a} and the identity \prettyref{eq:4-7}
were proved in \cite{MR2336583}. Here they arise from the minimal
Stinespring representation in \prettyref{thm:4-1}. More importantly,
this representation identifies $\mathcal{H}\left(f\right)$ as the
Kraus multiplicity space of $\Phi_{f}$ and the weighted composition
operators as its coordinate Kraus operators. So it gives the full
formula \prettyref{eq:4-6} for every $S\in B\left(H^{2}\right)$,
with \prettyref{eq:4-7} as the special case $S=I$.
\begin{prop}
\label{prop:4-4} For every $u\in\mathcal{H}\left(f\right)$, the
operator $W_{u,f}$ is bounded on $H^{2}$ and 
\begin{equation}
\left\Vert W_{u,f}\right\Vert \leq\left\Vert u\right\Vert _{\mathcal{H}\left(f\right)}.\label{eq:4-6-a}
\end{equation}
If $\left\{ u_{\alpha}\right\} _{\alpha\geq1}$ is an orthonormal
basis of $\mathcal{H}\left(f\right)$, then 
\begin{equation}
\Phi_{f}\left(S\right)=\sum_{\alpha\geq1}W_{u_{\alpha},f}SW^{*}_{u_{\alpha},f},\qquad S\in B\left(H^{2}\right),\label{eq:4-6}
\end{equation}
and 
\begin{equation}
\sum_{\alpha\geq1}W_{u_{\alpha},f}W^{*}_{u_{\alpha},f}=I.\label{eq:4-7}
\end{equation}
Both sums converge in the strong operator topology. 
\end{prop}

\begin{proof}
Define 
\[
A_{u}=\left(I_{H^{2}}\otimes\left\langle u,\cdot\right\rangle _{\mathcal{H}\left(f\right)}\right)V_{f}.
\]
Then 
\begin{equation}
\left\Vert A_{u}\right\Vert \leq\left\Vert u\right\Vert _{\mathcal{H}\left(f\right)}.\label{eq:nA}
\end{equation}
For $w\in\mathbb{D}$, 
\[
A_{u}k_{w}=\left\langle u,c_{f,w}\right\rangle k_{f\left(w\right)}=\overline{u\left(w\right)}k_{f\left(w\right)}.
\]
It follows that 
\[
A^{*}_{u}h=u\left(h\circ f\right)
\]
for $h\in H^{2}$. Thus $W_{u,f}=A^{*}_{u}$ and the norm estimate
\prettyref{eq:nA} follows.

Let $P_{N}$ be the projection onto the span of $u_{1},\ldots,u_{N}$.
Expanding the second tensor factor gives 
\[
\sum^{N}_{\alpha=1}W_{u_{\alpha},f}SW^{*}_{u_{\alpha},f}=V^{*}_{f}\left(S\otimes P_{N}\right)V_{f}.
\]
Since $P_{N}$ converges strongly to $I_{\mathcal{H}\left(f\right)}$,
the right hand side converges strongly to $\Phi_{f}\left(S\right)$.
This proves \prettyref{eq:4-6}. Taking $S=I$ gives \prettyref{eq:4-7}. 
\end{proof}

For a normal completely positive map on $B\left(H^{2}\right)$, its
Kraus rank means the Hilbert space dimension of the multiplicity space
in a minimal normal Stinespring representation.
\begin{cor}
\label{cor:4-5} The Kraus rank of $\Phi_{f}$ is 
\[
\dim\mathcal{H}\left(f\right).
\]
\end{cor}

\begin{proof}
The Stinespring representation in \prettyref{thm:4-1} is minimal,
and its multiplicity space is $\mathcal{H}\left(f\right)$. 
\end{proof}

The two kernels also give a direct extremality argument.
\begin{thm}
\label{thm:4-6} For every holomorphic map $f:\mathbb{D}\to\mathbb{D}$,
the map $\Phi_{f}$ is an extreme point of the convex set of unital
completely positive maps from $B\left(H^{2}\right)$ to $B\left(H^{2}\right)$.
Specifically, 
\begin{equation}
\left\{ \Phi_{f}:f\in\mathcal{S}\left(\mathbb{D}\right)\right\} \subsetneq\text{Ext}\left(\text{UCP}\left(\mathscr{B}\left(H^{2}\right)\right)\right).\label{eq:ext1}
\end{equation}
\end{thm}

\begin{proof}
The commutant of the minimal Stinespring representation is 
\[
\left\{ I_{H^{2}}\otimes T:T\in B\left(\mathcal{H}\left(f\right)\right)\right\} .
\]
By the extremality criterion for unital completely positive maps \cite{MR253059,MR1407488,MR1976867},
it is enough to show that 
\[
V^{*}_{f}\left(I_{H^{2}}\otimes T\right)V_{f}=0
\]
implies $T=0$.

For $z,w\in\mathbb{D}$, 
\[
\left\langle k_{z},V^{*}_{f}\left(I_{H^{2}}\otimes T\right)V_{f}k_{w}\right\rangle =k\left(f\left(z\right),f\left(w\right)\right)\left\langle c_{f,z},Tc_{f,w}\right\rangle _{\mathcal{H}\left(f\right)}.
\]
The Szegő kernel does not vanish on $\mathbb{D}\times\mathbb{D}$.
Therefore 
\[
\left\langle c_{f,z},Tc_{f,w}\right\rangle =0
\]
for all $z,w\in\mathbb{D}$. Since the kernel functions span $\mathcal{H}\left(f\right)$,
we have $T=0$. 

The inclusion \prettyref{eq:ext1} is proper, for example, because
$S\mapsto M^{*}_{z}SM_{z}$ is extreme but is not equal to $\Phi_{f}$
for any $f\in\mathcal{S}\left(\mathbb{D}\right)$.
\end{proof}

For inner functions, the associated endomorphisms and their Cuntz
families are studied in \cite{MR4456098}. The next result places
that case within the family $\left\{ \Phi_{f}\right\} $.
\begin{thm}
\label{thm:4-7} Let $f:\mathbb{D}\to\mathbb{D}$ be holomorphic.
\begin{enumerate}
\item The map $\Phi_{f}$ is a unital endomorphism of $B\left(H^{2}\right)$
if and only if $f$ is inner.
\item The map $\Phi_{f}$ has finite Kraus rank if and only if $f$ is a
finite Blaschke product. In that case, 
\[
\text{Kraus rank}\left(\Phi_{f}\right)=\deg\left(f\right).
\]
\item The map $\Phi_{f}$ is an automorphism of $B\left(H^{2}\right)$ if
and only if $f$ is a disk automorphism. 
\end{enumerate}
\end{thm}

\begin{proof}
Suppose first that $f$ is inner. Then 
\[
\mathcal{H}\left(f\right)=H^{2}\ominus fH^{2}
\]
with equality of Hilbert space norms \cite{MR1289670}. Multiplication
by $f$ is a pure isometry, and its Wold decomposition gives 
\[
H^{2}=\bigoplus^{\infty}_{n=0}f^{n}\mathcal{H}\left(f\right).
\]
Define 
\[
U_{f}:H^{2}\otimes\mathcal{H}\left(f\right)\longrightarrow H^{2}
\]
by 
\[
U_{f}\left(e_{n}\otimes u\right)=f^{n}u.
\]
The preceding orthogonal decomposition shows that $U_{f}$ is unitary.
For $w\in\mathbb{D}$, 
\[
U^{*}_{f}k_{w}=k_{f\left(w\right)}\otimes c_{f,w}=V_{f}k_{w}.
\]
Thus $V_{f}=U^{*}_{f}$ and 
\[
\Phi_{f}\left(S\right)=U_{f}\left(S\otimes I_{\mathcal{H}\left(f\right)}\right)U^{*}_{f}.
\]
It follows that $\Phi_{f}$ is an endomorphism.

Conversely, suppose that $\Phi_{f}$ is an endomorphism. Since $M^{*}_{z}M_{z}=I$,
\prettyref{eq:4-3} gives 
\[
\begin{aligned}I & =\Phi_{f}\left(M^{*}_{z}M_{z}\right)\\
 & =\Phi_{f}\left(M_{z}\right)^{*}\Phi_{f}\left(M_{z}\right)=M^{*}_{f}M_{f}.
\end{aligned}
\]
Hence multiplication by $f$ is an isometry on $H^{2}$, so $f$ is
inner. This proves the first assertion.

By \prettyref{cor:4-5}, the Kraus rank of $\Phi_{f}$ is $\dim\mathcal{H}\left(f\right)$.
The finite-dimensional characterization of de Branges-Rovnyak spaces
gives 
\[
\dim\mathcal{H}\left(f\right)<\infty
\]
if and only if $f$ is a finite Blaschke product, and in that case
\[
\dim\mathcal{H}\left(f\right)=\deg\left(f\right).
\]
See \cite{MR1289670,MR4456098}. 

If $f$ is a disk automorphism, then it is a Blaschke product of degree
one. The unitary $U_{f}$ above shows that $\Phi_{f}$ is an automorphism.

Conversely, suppose that $\Phi_{f}$ is an automorphism. It is then
an endomorphism, so $f$ is inner. An automorphism of $B\left(H^{2}\right)$
has Kraus rank one. Hence 
\[
\dim\mathcal{H}\left(f\right)=1,
\]
and $f$ is a Blaschke product of degree one. Thus $f$ is a disk
automorphism. 
\end{proof}

The family $\left\{ \Phi_{f}\right\} $ is compatible with composition.
Let $f,g:\mathbb{D}\to\mathbb{D}$ be holomorphic. A direct calculation
gives 
\begin{equation}
c_{g\circ f}\left(z,w\right)=c_{f}\left(z,w\right)c_{g}\left(f\left(z\right),f\left(w\right)\right).\label{eq:4-8}
\end{equation}

In terms of the second kernel, this identity says that the Kraus multiplicity
space for a composition embeds naturally into the tensor product of
the two multiplicity spaces.
\begin{prop}
\label{prop:4-8} There is a canonical isometry 
\[
J_{f,g}:\mathcal{H}\left(g\circ f\right)\longrightarrow\mathcal{H}\left(g\right)\otimes\mathcal{H}\left(f\right)
\]
defined on kernel functions by 
\begin{equation}
J_{f,g}c_{g\circ f,w}=c_{g,f\left(w\right)}\otimes c_{f,w}.\label{eq:4-9}
\end{equation}
Moreover, 
\begin{equation}
\Phi_{f}\circ\Phi_{g}=\Phi_{g\circ f}.\label{eq:4-10}
\end{equation}
\end{prop}

\begin{proof}
For $z,w\in\mathbb{D}$, \prettyref{eq:4-8} gives 
\[
\left\langle c_{g,f\left(z\right)}\otimes c_{f,z},c_{g,f\left(w\right)}\otimes c_{f,w}\right\rangle =c_{g}\left(f\left(z\right),f\left(w\right)\right)c_{f}\left(z,w\right)=c_{g\circ f}\left(z,w\right).
\]
Thus \prettyref{eq:4-9} extends to an isometry.

On the kernel functions of $H^{2}$, 
\begin{equation}
\left(V_{g}\otimes I_{\mathcal{H}\left(f\right)}\right)V_{f}=\left(I_{H^{2}}\otimes J_{f,g}\right)V_{g\circ f}.\label{eq:4-11}
\end{equation}
Indeed, both sides send $k_{w}$ to 
\[
k_{g\left(f\left(w\right)\right)}\otimes c_{g,f\left(w\right)}\otimes c_{f,w}.
\]
Using \prettyref{eq:4-11} in the Stinespring formulas gives \prettyref{eq:4-10}. 
\end{proof}

\begin{cor}
\label{cor:4-9} For $n\geq1$, 
\begin{equation}
c_{f^{\circ n}}\left(z,w\right)=\prod^{n-1}_{j=0}c_{f}\left(f^{\circ j}\left(z\right),f^{\circ j}\left(w\right)\right).\label{eq:4-12}
\end{equation}
Moreover, 
\[
\Phi^{n}_{f}=\Phi_{f^{\circ n}}.
\]
There is an isometry 
\[
J_{f,n}:\mathcal{H}\left(f^{\circ n}\right)\longrightarrow\mathcal{H}\left(f\right)^{\otimes n}
\]
given on kernel functions by 
\[
J_{f,n}c_{f^{\circ n},w}=c_{f,f^{\circ\left(n-1\right)}\left(w\right)}\otimes\cdots\otimes c_{f,f\left(w\right)}\otimes c_{f,w}.
\]
\end{cor}

\begin{proof}
Apply \prettyref{prop:4-8} repeatedly. 
\end{proof}

\begin{example}
\label{exa:4-10} Let 
\[
f\left(z\right)=z^{n},\qquad n\geq2.
\]
Then 
\[
c_{f}\left(z,w\right)=\frac{1-z^{n}\overline{w}^{n}}{1-z\overline{w}}=\sum^{n-1}_{j=0}z^{j}\overline{w}^{j}.
\]
Thus 
\[
\mathcal{H}\left(f\right)=\text{span}\left\{ 1,z,\ldots,z^{n-1}\right\} .
\]
Let $e_{m}\left(z\right)=z^{m}$. The Stinespring isometry satisfies
\[
V_{f}e_{nq+j}=e_{q}\otimes e_{j}
\]
for $q\geq0$ and $0\leq j<n$.

For $0\leq j<n$, define 
\[
W_{j}h\left(z\right)=z^{j}h\left(z^{n}\right).
\]
Then 
\[
W_{j}e_{q}=e_{nq+j}.
\]
The operators $W_{0},\ldots,W_{n-1}$ are isometries with mutually
orthogonal ranges and 
\[
\sum^{n-1}_{j=0}W_{j}W^{*}_{j}=I.
\]
The Kraus representation is 
\[
\Phi_{f}\left(S\right)=\sum^{n-1}_{j=0}W_{j}SW^{*}_{j}.
\]
In particular, $\Phi_{f}$ is a unital endomorphism of Kraus rank
$n$. 
\end{example}

\section{Kernel dynamics}\label{sec:5}

The composition formula from \prettyref{prop:4-8} makes it natural
to ask what happens under repeated application of $\Phi_{f}$. For
background on composition operators and the iteration of holomorphic
self-maps, see, for example, \cite{MR1397026}.

Strong mixing and asymptotic stability for unital completely positive
maps have been studied in a general operator-algebraic setting \cite{MR2060791,MR2329688},
but here the two-kernel view gives a direct description of the iterates.
We show below that the Szegő kernel vectors follow the orbit of $f$,
and the de Branges-Rovnyak kernel supplies the scalar coefficient.
When the iterates of $f$ converge to a point $a\in\mathbb{D}$, this
formula shows that the preadjoint iterates of $\Phi_{f}$ converge
in trace norm to the map 
\[
\rho\longmapsto\text{Tr}\left(\rho\right)\rho_{a}.
\]
So every normal state converges to the normalized Szegő kernel state
at $a$. This trace-norm mixing result for the family $\left\{ \Phi_{f}\right\} $
does not seem to have appeared in this kernel form. 

Let $\mathcal{S}_{1}\left(H^{2}\right)$ denote the trace-class operators
on $H^{2}$. For $z,w\in\mathbb{D}$, put 
\[
\theta_{w,z}=\left|k_{w}\right\rangle \left\langle k_{z}\right|.
\]

\begin{prop}
\label{prop:5-1} For every $z,w\in\mathbb{D}$, 
\begin{equation}
\Phi_{f,*}\left(\theta_{w,z}\right)=c_{f}\left(z,w\right)\left|k_{f\left(w\right)}\right\rangle \left\langle k_{f\left(z\right)}\right|.\label{eq:5-1}
\end{equation}
Consequently, for every $n\geq1$, 
\begin{equation}
\Phi^{n}_{f,*}\left(\theta_{w,z}\right)=c_{f^{\circ n}}\left(z,w\right)\left|k_{f^{\circ n}\left(w\right)}\right\rangle \left\langle k_{f^{\circ n}\left(z\right)}\right|.\label{eq:5-2}
\end{equation}
\end{prop}

\begin{proof}
For $S\in B\left(H^{2}\right)$, we have 
\[
\begin{aligned}\text{Tr}\left(\Phi_{f,*}\left(\theta_{w,z}\right)S\right) & =\left\langle k_{z},\Phi_{f}\left(S\right)k_{w}\right\rangle _{H^{2}}\\
 & =\left\langle V_{f}k_{z},\left(S\otimes I_{\mathcal{H}\left(f\right)}\right)V_{f}k_{w}\right\rangle \\
 & =\left\langle k_{f\left(z\right)},Sk_{f\left(w\right)}\right\rangle _{H^{2}}\left\langle c_{f,z},c_{f,w}\right\rangle _{\mathcal{H}\left(f\right)}\\
 & =c_{f}\left(z,w\right)\left\langle k_{f\left(z\right)},Sk_{f\left(w\right)}\right\rangle _{H^{2}}.
\end{aligned}
\]
This proves \prettyref{eq:5-1}. Iterating the formula gives 
\[
\begin{aligned}\Phi^{n}_{f,*}\left(\theta_{w,z}\right) & =\left(\prod^{n-1}_{j=0}c_{f}\left(f^{\circ j}\left(z\right),f^{\circ j}\left(w\right)\right)\right)\cdot\left|k_{f^{\circ n}\left(w\right)}\right\rangle \left\langle k_{f^{\circ n}\left(z\right)}\right|.\end{aligned}
\]
The product formula in \prettyref{eq:4-12} now gives \prettyref{eq:5-2}. 
\end{proof}

\begin{thm}
\label{thm:5-2} Suppose that there is a point $a\in\mathbb{D}$ such
that 
\[
f^{\circ n}\longrightarrow a
\]
locally uniformly on $\mathbb{D}$. Then, for every $\rho\in\mathcal{S}_{1}\left(H^{2}\right)$,
\begin{equation}
\lim_{n\rightarrow\infty}\left\Vert \Phi^{n}_{f,*}\left(\rho\right)-\text{Tr}\left(\rho\right)\rho_{a}\right\Vert _{1}=0,\label{eq:5-3}
\end{equation}
where 
\[
\rho_{a}=\left|\widetilde{k}_{a}\right\rangle \left\langle \widetilde{k}_{a}\right|.
\]
Consequently, for every $S\in B\left(H^{2}\right)$, 
\begin{equation}
\Phi^{n}_{f}\left(S\right)\longrightarrow\left\langle \widetilde{k}_{a},S\widetilde{k}_{a}\right\rangle _{H^{2}}I\label{eq:5-4}
\end{equation}
in the ultraweak topology. 
\end{thm}

\begin{proof}
Fix $z,w\in\mathbb{D}$. Since 
\[
c_{f^{\circ n}}\left(z,w\right)=\frac{1-f^{\circ n}\left(z\right)\overline{f^{\circ n}\left(w\right)}}{1-z\overline{w}},
\]
we have 
\[
c_{f^{\circ n}}\left(z,w\right)\longrightarrow\left(1-\left|a\right|^{2}\right)k\left(z,w\right).
\]
The map $w\mapsto k_{w}$ from $\mathbb{D}$ into $H^{2}$ is norm
continuous. Hence 
\[
\left|k_{f^{\circ n}\left(w\right)}\right\rangle \left\langle k_{f^{\circ n}\left(z\right)}\right|\longrightarrow\left|k_{a}\right\rangle \left\langle k_{a}\right|
\]
in trace norm. It follows from \prettyref{eq:5-2} that 
\[
\Phi^{n}_{f,*}\left(\theta_{w,z}\right)\longrightarrow k\left(z,w\right)\rho_{a}
\]
in trace norm, since 
\[
\rho_{a}=\left(1-\left|a\right|^{2}\right)\left|k_{a}\right\rangle \left\langle k_{a}\right|.
\]
Moreover, 
\[
\text{Tr}\left(\theta_{w,z}\right)=\left\langle k_{z},k_{w}\right\rangle _{H^{2}}=k\left(z,w\right).
\]
Thus \prettyref{eq:5-3} holds on the linear span of the operators
$\theta_{w,z}$.

The kernel functions span a dense subspace of $H^{2}$. It follows
that 
\[
\mathcal{D}=\text{span}\left\{ \theta_{w,z}:z,w\in\mathbb{D}\right\} 
\]
is dense in $\mathcal{S}_{1}\left(H^{2}\right)$.

Each $\Phi^{n}_{f,*}$ is a contraction in trace norm. The map 
\[
L\left(\rho\right)=\text{Tr}\left(\rho\right)\rho_{a}
\]
is also a contraction. Given $\rho\in\mathcal{S}_{1}\left(H^{2}\right)$
and $\sigma\in\mathcal{D}$, we have 
\[
\begin{aligned}\left\Vert \Phi^{n}_{f,*}\left(\rho\right)-L\left(\rho\right)\right\Vert _{1} & \leq\left\Vert \Phi^{n}_{f,*}\left(\rho-\sigma\right)\right\Vert _{1}+\left\Vert \Phi^{n}_{f,*}\left(\sigma\right)-L\left(\sigma\right)\right\Vert _{1}+\left\Vert L\left(\sigma-\rho\right)\right\Vert _{1}\\
 & \leq2\left\Vert \rho-\sigma\right\Vert _{1}+\left\Vert \Phi^{n}_{f,*}\left(\sigma\right)-L\left(\sigma\right)\right\Vert _{1}.
\end{aligned}
\]
Density of $\mathcal{D}$ now proves \prettyref{eq:5-3}.

Finally, for $\rho\in\mathcal{S}_{1}\left(H^{2}\right)$ and $S\in B\left(H^{2}\right)$,
\[
\begin{aligned}\text{Tr}\left(\rho\Phi^{n}_{f}\left(S\right)\right) & =\text{Tr}\left(\Phi^{n}_{f,*}\left(\rho\right)S\right)\\
 & \longrightarrow\text{Tr}\left(\rho\right)\text{Tr}\left(\rho_{a}S\right)=\text{Tr}\left(\rho\right)\left\langle \widetilde{k}_{a},S\widetilde{k}_{a}\right\rangle _{H^{2}}.
\end{aligned}
\]
This proves \prettyref{eq:5-4}. 
\end{proof}

\begin{cor}
\label{cor:5-3} Under the assumptions of \prettyref{thm:5-2}, 
\begin{equation}
\left\{ S\in B\left(H^{2}\right):\Phi_{f}\left(S\right)=S\right\} =\mathbb{C}I.\label{eq:5-5}
\end{equation}
Moreover, $\rho_{a}$ is the unique normal state invariant under $\Phi_{f,*}$. 
\end{cor}

\begin{proof}
If $\Phi_{f}\left(S\right)=S$, then 
\[
S=\Phi^{n}_{f}\left(S\right)
\]
for every $n\geq1$. By \prettyref{eq:5-4}, 
\[
S=\left\langle \widetilde{k}_{a},S\widetilde{k}_{a}\right\rangle _{H^{2}}I.
\]
Conversely, every scalar multiple of $I$ is fixed because $\Phi_{f}$
is unital.

The convergence of the iterates also implies that $f\left(a\right)=a$.
Indeed, 
\[
f^{\circ\left(n+1\right)}=f\circ f^{\circ n}\longrightarrow f\left(a\right),
\]
while the left hand side converges to $a$. Hence $f\left(a\right)=a$,
and \prettyref{eq:4-4} gives 
\[
\Phi_{f,*}\left(\rho_{a}\right)=\rho_{a}.
\]

If $\rho$ is any normal invariant state, then 
\[
\rho=\Phi^{n}_{f,*}\left(\rho\right)\longrightarrow\rho_{a}
\]
in trace norm by \prettyref{thm:5-2}. Therefore $\rho=\rho_{a}$. 
\end{proof}

\begin{rem}
The assignment 
\[
f\longmapsto\Phi_{f}
\]
is injective. Indeed, if $\Phi_{f}=\Phi_{g}$, then 
\[
M_{f}=\Phi_{f}\left(M_{z}\right)=\Phi_{g}\left(M_{z}\right)=M_{g},
\]
and hence $f=g$. Together with \prettyref{prop:4-8}, this gives
a faithful reversed-composition representation of the semigroup of
holomorphic self-maps of $\mathbb{D}$ by extreme unital completely
positive maps on $B\left(H^{2}\right)$. 
\end{rem}

\bibliographystyle{amsalpha}
\bibliography{ref}

\providecommand{\bysame}{\leavevmode\hbox to3em{\hrulefill}\thinspace}
\providecommand{\MR}{\relax\ifhmode\unskip\space\fi MR }
\providecommand{\MRhref}[2]{%
  \href{http://www.ams.org/mathscinet-getitem?mr=#1}{#2}
}
\providecommand{\href}[2]{#2}
\begin{thebibliography}{CMS12}

\bibitem[Aro50]{MR51437}
N.~Aronszajn, \emph{Theory of reproducing kernels}, Trans. Amer. Math. Soc.
  \textbf{68} (1950), 337--404. \MR{51437}

\bibitem[Arv69]{MR253059}
William~B. Arveson, \emph{Subalgebras of {$C\sp{\ast} $}-algebras}, Acta Math.
  \textbf{123} (1969), 141--224. \MR{253059}

\bibitem[Arv04]{MR2060791}
William Arveson, \emph{Asymptotic stability. {I}. {C}ompletely positive maps},
  Internat. J. Math. \textbf{15} (2004), no.~3, 289--312. \MR{2060791}

\bibitem[Arv07]{MR2329688}
\bysame, \emph{The asymptotic lift of a completely positive map}, J. Funct.
  Anal. \textbf{248} (2007), no.~1, 202--224. \MR{2329688}

\bibitem[CM95]{MR1397026}
Carl~C. Cowen and Barbara~D. MacCluer, \emph{Composition operators on spaces of
  analytic functions}, Studies in Advanced Mathematics, CRC Press, Boca Raton,
  FL, 1995. \MR{1397026}

\bibitem[CMS12]{MR4456098}
Dennis Courtney, Paul~S. Muhly, and Samuel~W. Schmidt, \emph{Composition
  operators and endomorphisms}, Complex Anal. Oper. Theory \textbf{6} (2012),
  no.~1, 163--188. \MR{4456098}

\bibitem[FM97]{MR1407488}
Douglas~R. Farenick and Phillip~B. Morenz, \emph{{$C^*$}-extreme points in the
  generalised state spaces of a {$C^*$}-algebra}, Trans. Amer. Math. Soc.
  \textbf{349} (1997), no.~5, 1725--1748. \MR{1407488}

\bibitem[Jur07]{MR2336583}
Michael~T. Jury, \emph{Reproducing kernels, de {B}ranges-{R}ovnyak spaces, and
  norms of weighted composition operators}, Proc. Amer. Math. Soc. \textbf{135}
  (2007), no.~11, 3669--3675. \MR{2336583}

\bibitem[Kra71]{MR292434}
K.~Kraus, \emph{General state changes in quantum theory}, Ann. Physics
  \textbf{64} (1971), 311--335. \MR{292434}

\bibitem[Pau02]{MR1976867}
Vern Paulsen, \emph{Completely bounded maps and operator algebras}, Cambridge
  Studies in Advanced Mathematics, vol.~78, Cambridge University Press,
  Cambridge, 2002. \MR{1976867}

\bibitem[Ped58]{MR2938971}
George Pedrick, \emph{Theory of reproducing kernels for {H}ilbert spaces of
  vector valued functions}, ProQuest LLC, Ann Arbor, MI, 1958, Thesis
  (Ph.D.)--University of Kansas. \MR{2938971}

\bibitem[Sar94]{MR1289670}
Donald Sarason, \emph{Sub-{H}ardy {H}ilbert spaces in the unit disk},
  University of Arkansas Lecture Notes in the Mathematical Sciences, vol.~10,
  John Wiley \& Sons, Inc., New York, 1994, A Wiley-Interscience Publication.
  \MR{1289670}

\bibitem[Sti55]{MR69403}
W.~Forrest Stinespring, \emph{Positive functions on {$C^*$}-algebras}, Proc.
  Amer. Math. Soc. \textbf{6} (1955), 211--216. \MR{69403}

\bibitem[Sza21]{MR4250453}
Franciszek~Hugon Szafraniec, \emph{Revitalising {P}edrick's approach to
  reproducing kernel {H}ilbert spaces}, Complex Anal. Oper. Theory \textbf{15}
  (2021), no.~4, Paper No. 66, 12. \MR{4250453}

\end{thebibliography}

\end{document}